\documentclass[12pt]{article}
\usepackage{latexsym,amsmath,amssymb}
\usepackage[latin1]{inputenc}
\usepackage {amsthm}
\usepackage{times}
\usepackage{amscd}
\usepackage{epsf}
\usepackage{graphicx}
\usepackage{graphics}

\newcommand{\fl}{\longrightarrow}
\newfont{\bb}{msbm10 at 12pt}
\def\r{\hbox{\bb R}}
\def\h{\hbox{\bb H}}
\def\c{\hbox{\bb C}}
\def\s{\hbox{\bb S}}
\def\mm{\mathbb{M}\times\mathbb{R}}
\def\m{\mathbb{M}}
\def\mi{\mathbb{M}(\infty)}
\def\G{\Gamma}
\def\g{\gamma}
\def\P{{\cal P}}
\def\Si{\Sigma}

\newcommand{\ee}{\begin{equation}}

\newcommand{\fe}{\end{equation}}

\newcommand{\me}[2]{\langle #1,#2 \rangle }

\usepackage[latin1]{inputenc}
\usepackage {amsmath}
\usepackage {amsthm}
\usepackage{times}
\usepackage{amscd}
\usepackage{epsf}

\numberwithin{equation} {section}

\begin{document}

\theoremstyle{plain}\newtheorem{lem}{Lemma}[section]
\theoremstyle{plain}\newtheorem{pro}{Proposition}[section]
\theoremstyle{plain}\newtheorem{teo}{Theorem}[section]
\theoremstyle{plain}\newtheorem{eje}{Example}[section]
\theoremstyle{plain}\newtheorem{no}{Remark}[section]
\theoremstyle{plain}\newtheorem{cor}{Corollary}[section]

\begin{center}
\rule{15cm}{1.5pt} \vspace{.6cm}

{\Large \bf Minimal surfaces and harmonic diffeomorphisms \\[3mm] from the complex plane
onto a Hadamard surface.} \vspace{0.4cm}

\rule{15cm}{1.5pt}\vspace{.5cm}

{\bf José A. Gálvez$^a$\footnote{The first author is partially supported by
MEC-FEDER, Grant No MTM2007-65249.}\quad and\quad Harold Rosenberg$^b$}
\end{center}

\noindent{\small ${}^a$Departamento de Geometr\'\i
a y Topolog\'\i a, F. Ciencias, Universidad de Granada, 18071 Granada, Spain \\
e-mail: jagalvez@ugr.es\\
${}^b$  Institut de Mathématiques, Université Paris VII, 2 place Jussieu, 75005
Paris, France\\ e-mail: rosen@math.jussieu.fr} \vspace{8mm}

\noindent{\small {\bf Abstract.} We construct harmonic diffeomorphisms from the
complex plane $\c$ onto any Hadamard surface $\m$ whose curvature is bounded above
by a negative constant. For that, we prove a Jenkins-Serrin type theorem for minimal
graphs in $\mm$ over domains of $\m$ bounded by ideal geodesic polygons and show the
existence of a sequence of minimal graphs over polygonal domains converging to an
entire minimal graph in $\mm$ with the conformal structure of $\c$.}

\section{Introduction.}

There are many harmonic diffeomorphisms from the complex plane $\c$ onto the
hyperbolic plane $\h$. They were constructed by finding entire minimal graphs in
$\h\times\r$ whose conformal type is $\c$ \cite{CR}. The vertical projection of such
a graph onto $\h$ is such a harmonic diffeomorphism. It was conjectured that there
was no such map \cite{SY}.

In this paper we will show there are harmonic diffeomorphisms from $\c$ onto any
Hadamard surface whose curvature is bounded above by a negative constant. The
question of their existence was posed by R. Schoen.

We proceed as in \cite{CR} by constructing entire minimal graphs in $\mm$, of
conformal type $\c$; $\m$ a complete simply connected Riemannian surface with
curvature $K_{\m}\leq a<0$. The construction of these graphs in $\h\times\r$ can be
done in $\mm$; the geometry of the asymptotic boundary of $\m$ is sufficiently close
to that of $\h$.

We are thus able to prove a Jenkins-Serrin type theorem for minimal graphs in $\mm$,
over domains of $\m$ bounded by ideal geodesic polygons. There are several
constructions in our paper, which we believe will be useful for future research.

An interesting question is whether our theorems hold when $K_{\m}<0.$

\section{Preliminaries.}

We will devote this Section to present some basic properties of Hadamard manifolds,
which will be necessary for our study (see for instance \cite{E1,E2} for details).

Let $\m$ be a Hadamard manifold, that is, a complete simply connected Riemannian
manifold with non positive sectional curvature. It is classically known that there
is a unique geodesic joining two points of $\m$. Thus, the concept of (geodesic)
convexity is naturally defined for sets in $\m$.

We say that two geodesics $\g_1(t),\g_2(t)$ of $\m$, parametrized by arc length, are
asymptotic if there exists a constant $c>0$ such that the distance
$d(\g_1(t),\g_2(t))$ is less than $c$ for all $t\geq0$. Analogously, two unit
vectors $v_1,v_2$ are said to be asymptotic if the corresponding geodesics
$\g_{v_1}(t),\g_{v_2}(t)$ have this property.

To be asymptotic is an equivalence relation on the oriented unit speed geodesics or
on the set of unit vectors of $\m$. Each one of these equivalence classes will be
called a point at infinity, and $\mi$ will denote the set of points at infinity.

We will denote by $\g(+\infty)$ or $v(\infty)$ the equivalence class of the
corresponding geodesic $\g(t)$ or unit vector $v$.

When $\m$ is a Hadamard manifold with sectional curvature bounded from above by a
negative constant then any two asymptotic geodesics $\g_1,\g_2$ satisfy that the
distance between the two curves ${\g_1}_{|[t_0,+\infty)},{\g_2}_{|[t_0,+\infty)}$ is
zero for any $t_0\in\r$. In addition, under this curvature hypothesis, given
$x,y\in\mi$ there exists a unique oriented unit speed geodesic $\g$ such that
$\g(+\infty)=x$ and $\g(-\infty)=y$, where $\g(-\infty)$ is the corresponding point
at infinity when we change the orientation of $\g$.

For any point $p$ of a general Hadamard manifold, there is a bijective
correspondence between the set of unit vectors at $p$ and $\mi$, where a unit vector
$v$ is mapped to the point at infinity $v(\infty)$. Equivalently, given a point
$p\in\m$ and a point $x\in\mi$, there exists a unique oriented unit speed geodesic
$\g$ such that $\g(0)=p$ and $\g(+\infty)=x$. In particular, $\mi$ is bijective to a
sphere.

In fact, there exists a topology on $\m^\ast=\m\cup\mi$ satisfying
\begin{enumerate}
\item the restriction to $\m$ agrees with the topology induced by the Riemannian
distance,
\item there exists a homeomorphism from $\m^\ast$ onto the closed unit ball which
identifies $\mi$ with the unit sphere,
\item the map $v\rightarrow v(\infty)$ is a homeomorphism from the unit sphere of the
tangent plane at a fixed point $p$ onto $\mi$.
\end{enumerate}

This topology is called the cone topology of $\m^\ast$ and can be obtained as
follows. Let $p\in\m$ and ${\cal U}$ an open set in the unit sphere of its tangent
plane. Define for any $r>0$
$$
T({\cal U},r)=\{\g_v(t)\in\m^\ast:\ v\in{\cal U},\ r<t\leq+\infty\}.
$$
The cone topology is the unique one such that its restriction to $\m$ is the
topology induced by the Riemannian distance and such that the sets $T({\cal U},r)$
containing a point $x\in\mi$ form a neighborhood basis at $x$.

Given a set $A\subseteq\m$, we denote by $\partial_{\infty}A$ the set $\partial
A\cap\mi$, where $\partial A$ is the boundary of $A$ for the cone topology.

Horospheres are defined in terms of Busemann functions. Given a unit vector $v$, the
Busemann function $B_v:\m\fl\r$, associated to $v$, is
$$
B_v(p)=\lim_{t\rightarrow+\infty} d(p,\g_v(t))-t.
$$
This function verifies some important properties
\begin{enumerate}
\item $B_v$ is a ${\cal C}^2$ convex function on $\m$,
\item the gradient $\nabla B_v(p)$ is the unique unit vector $w$ at $p$ such that
$v(\infty)=-w(\infty)$,
\item if $w$ is a unit vector such that $v(\infty)=w(\infty)$ then $B_v-B_w$ is a
constant function on $\m$.
\end{enumerate}

Given a point $x\in\mi$ and a unit vector $v$ such that $v(\infty)=x$ we define the
horospheres at $x$ as the level sets of the Busemann function $B_v$. By property 3,
the horospheres at $x$ do not depend on the choice of $v$. The horospheres at a
point $x\in\mi$ give a foliation of $\m$ and, from property one, each one bounds a
convex domain in $\m$ called a horoball. Moreover, the intersection between a
geodesic $\g$ and a horosphere at $\g(+\infty)$ is always orthogonal from property
two.

With respect to distance from horospheres we present the following facts.
\begin{enumerate}
\item Let $p\in\m$, $H_x$ a horosphere at $x$ and $\g$ the geodesic passing through $p$
having $x$ as a point at infinity, then $H_x\cap\g$ is the closest point on $H_x$ to
$p$.
\item If $\g$ is a geodesic with points at infinity $x,y$, and $H_x$, $H_y$ are
disjoint horospheres at these points then the distance between $H_x$ and $H_y$
agrees with the distance between the points $H_x\cap\g$ and $H_y\cap\g$.
\item The function ${\cal D}: \m\times \m^\ast\times \m\fl\r$ given by
\ee
\label{D}
{\cal D}(a,b,c)=\left\{
\begin{array}
{ll} d(c,b)-d(a,b)&\mbox{if }b\in \m\\
B_v(c)
&\mbox{if }b\in \m(\infty)
\end{array} \right.
\fe is continuous, where $v$ is the unique unit tangent vector at $a$ such that
$v(\infty)=b$. ${\cal D}(a,b,c)$ measures the diference between the oriented
distance from $a$ and $c$ to any horosphere at $b\in\mi$. In particular, ${\cal
D}(a,b,c)<0$ means that $c$ is in the horoball whose boundary is the horosphere at
$b$ passing across $a$.
\end{enumerate}

\section{A Jenkins-Serrin type theorem for ideal polygons.}

From now on we will assume $\m$ is a simply connected, complete surface with Gauss
curvature bounded from above by a negative constant.

We say that $\Gamma$ is an {\it ideal polygon} if $\Gamma$ is a Jordan curve in
$\m^\ast$ which is a geodesic polygon with an even number of sides and all the
vertices in $\mi$. As usual, we will denote by $A_1, B_1,\ldots,A_k,B_k$ the sides
of $\G$, which are oriented counter-clockwise.

Now, we study the Dirichlet problem for the minimal surface equation in the domain
$D$ bounded by an ideal polygon $\G$. That is, we look for a solution $u:D\fl\r$ to
the equation
\begin{equation}\label{minimal}
\mbox{div}\left(\frac{\nabla u}{\sqrt{1+|\nabla u|^2}}\right)=0.
\end{equation}
Here, we prescribe the $+\infty$ data on each side $A_i$ and $-\infty$ on each side
$B_i$.

For relatively compact domains $D\subseteq \m$, it is well-known that there are
necessary and sufficient conditions on the lengths of the sides of polygons
inscribed in $\Gamma$ in order to solve this Dirichlet problem (see \cite{JS},
\cite{NR}, \cite{P}).

When $\G$ is an ideal polygon the length of each side is infinity and the previous
conditions make no sense. However, in \cite{CR}, the authors devise a manner to
compare the ``lengths'' of sides.

Fix an ideal polygon $\G$ and consider pairwise disjoint horocycles $H_i$ at each
vertex $a_i$ of $\G$.

For each side $A_i$, let us denote by $\widetilde{A}_i$ the compact geodesic arc
between the horocycles at the vertices of $A_i$, and by $|A_i|$ the length of
$\widetilde{A}_i$, that is, the distance between the horocycles. Analogously, one
defines $\widetilde{B}_i$ and $|B_i|$ for each side $B_i$, (cf. Figure 1).

\begin{figure}[h]
\mbox{}
\begin{center}
\includegraphics[height=5cm]{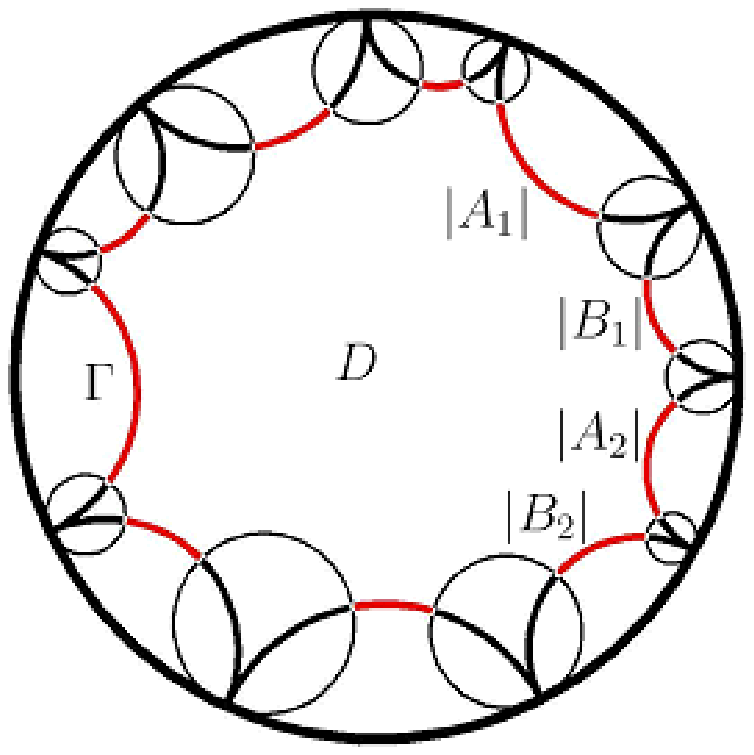}\\
Figure 1.
\end{center}
\end{figure}

Observe that if we define
$$
a(\G)=\sum_{i=1}^{k}|A_i|,\qquad b(\G)=\sum_{i=1}^{k}|B_i|,
$$
then $a(\G)-b(\G)$ does not depend on the choice of horocycles. This is due to the
fact that if we change a horocycle at a vertex then $a(\G)$ and $b(\G)$ increase or
decrease in the same quantity.

Let $D$ be the domain bounded by an ideal polygon $\G$. We say that a simple closed
geodesic polygon ${\cal P}$ is {\it inscribed} in $D$ if each vertex of $\P$ is a
vertex of $\G$.

Each side of $\P$ is one side $A_i$ or $B_i$ of $\G$, or a geodesic contained in $D$
(cf. Figure 2). Thus, the definition of $a(\G)$ and $b(\G)$ extends to $\P$. In
addition, we define the {\it truncated length} of the inscribed polygon $|\P|$ as
the sum of the lengths of the compact arcs of each side of $\P$ bounded by the
horocycles at its vertices.

\begin{figure}[h]
\mbox{}
\begin{center}
\includegraphics[height=5cm]{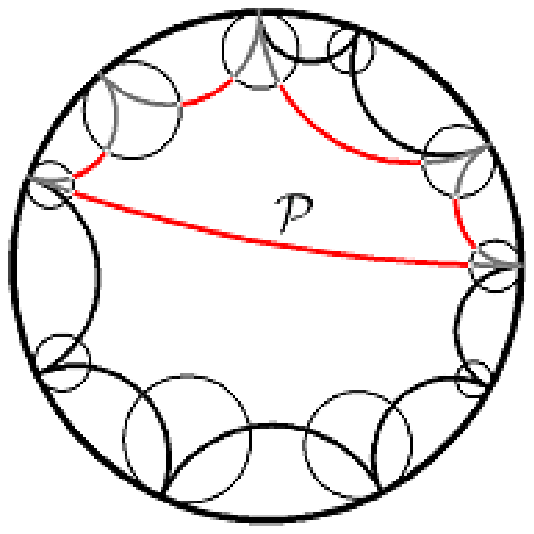}\\
Figure 2.
\end{center}
\end{figure}

Now, we can state a Jenkins-Serrin type theorem on domains of $\m$ bounded by an
ideal polygon $\G$.

\begin{teo}\label{JS}
There is a solution to the Dirichlet problem for the minimal surface equation in the
domain $D$ bounded by $\G$ with prescribed data $+\infty$ at $A_i$ and $-\infty$ at
$B_i$ if, and only if, the following two conditions are satisfied
\begin{enumerate}
\item $a(\G)-b(\G)=0$,
\item for all inscribed polygons $\P$ in $D$ different from $\G$ there exist
horocycles at the vertices such that
$$
2\,a(\P)<|\P|\qquad\mbox{and}\qquad 2\,b(\P)<|\P|.
$$
\end{enumerate}

Moreover, the solution is unique up to additive constants.
\end{teo}

\begin{no}\label{c2}
Notice that $a(\P)$ and $b(\P)$ depend on the chosen horocycles at the vertices.
However, if condition 2 is satisfied for a particular choice of horocycles then it
is also satisfied for all smaller horocycles at the vertices.

In addition, let $A_i, B_j$ be the two sides of $\G$ with a common vertex of $\P$.
If the side $A_i$ does not belong to $\P$ then $2\,a(\P)<|\P|$ is satisfied for the
choice of a small horocycle at the vertex. Thus, if $\P$ is an inscribed polygon in
$D$ such that there exists a vertex of $\P$ not containing the adjacent side $A_i$
of $\G$ and another vertex of $\P$ not containing the adjacent side $B_{i'}$ of
$\G$, then condition 2 is satisfied for the polygon $\P$.
\end{no}

\noindent {\it Proof of Theorem \ref{JS}.} This Theorem was proved in \cite{CR} when
$\m$ is the hyperbolic plane $\h$. The reader can check their proof works for a
general surface $\m$ once the existence of a Scherk type surface on each halfspace
of $\m$ is established.

This Scherk type surface in $\h$ is unique up to isometries of the ambient space and
was explicitly computed by U. Abresch and R. Sa Earp \cite{S}. For a general $\m$ we
now show its existence.

\begin{pro}\label{megascherk}
Let $\g$ be a complete geodesic in $\m$ and $\Omega$ a connected component of
$\m-\g$. There exists a positive solution $u$ to the Dirichlet problem for the
minimal surface equation in $\Omega$ with prescribed data $+\infty$ at $\g$ and such
that
$$
\lim \{u(p_n)\}=0
$$
for each sequence $\{p_n\}$ of points in $\Omega$ with distance to $\g$ going to
infinity.
\end{pro}
\begin{proof}
Since $\m$ is a Hadamard surface then
$$
\varphi(s,t)=\exp_{\g(t)}(s\ J\g'(t)), \qquad (s,t)\in\r^2
$$
is a global parametrization of $\m$, where the geodesic $\g(t)$ is parametrized by
arc length, $\exp$ is the usual exponential map and $J$ stands for the rotation in
$\m$ by $\pi/2$. In addition, we can assume that $\Omega$ is parametrized for $s>0$.

We observe that
$$
\langle\frac{\partial\ }{\partial s},\frac{\partial\ }{\partial s}\rangle=1,\qquad
\langle\frac{\partial\ }{\partial s},\frac{\partial\ }{\partial t}\rangle=0
$$
where $\langle,\rangle$ is the induced metric in $\m$. Moreover, if we denote by
$G(s,t)$ the function $\langle\frac{\partial\ }{\partial t},\frac{\partial\
}{\partial t}\rangle$ then
\begin{equation}
\label{3} G(0,t)=1,\quad G_s(0,t)=0,\qquad t\in\r
\end{equation}
since $\g(t)$ is a geodesic. Here, $G_s$ denotes the derivative of $G$ with respect
to $s$.

Now, we consider a graph $\psi(s,t)=(\varphi(s,t),h(s))$ on $\Omega$ which has
constant height for equidistant points to $\g$, that is, when $s$ is constant.

For the unit normal of the graph pointing down, the mean curvature of the immersion
is positive if and only if
\begin{equation}
\label{5} G_s\,h_s(1+h_s^2)+2G\,h_{ss}< 0,\qquad s>0,t\in\r.
\end{equation}

On the other hand, the Gauss curvature of $\m$ is given by
\begin{equation}
\label{1}
K(s,t)=-\frac{1}{4}\left(\frac{G_s}{G}\right)^2-\frac{1}{2}\left(\frac{G_s}{G}\right)_s.
\end{equation}

We notice that, for any constant $d<0$, the function
$\widetilde{G}(s)=\cosh^2(\sqrt{-d}\,s)$ verifies
\begin{equation}
\label{2}
d=-\frac{1}{4}\left(\frac{\widetilde{G}_s}{\widetilde{G}}\right)^2-
\frac{1}{2}\left(\frac{\widetilde{G}_s}{\widetilde{G}}\right)_s.
\end{equation}
Observe that $ds^2+\widetilde{G}(s)\,dt^2$ is the hyperbolic metric of curvature
$d$. Moreover, the function
\begin{equation}
\label{h}
\widetilde{h}(s)=-\frac{1}{\sqrt{-d}}\
\log\left(\tanh\left(\frac{1}{2}\,\sqrt{-d}\,s\right)\right), \qquad s>0,
\end{equation}
is decreasing and satisfies
\begin{equation}
\label{7}
\widetilde{G}_s\,\widetilde{h}_s(1+\widetilde{h}_s^2)+2\widetilde{G}\,\widetilde{h}_{ss}=
0.
\end{equation}
That is, $\widetilde{h}(s)$ is the minimal graph in the hyperbolic space found by
Abresch and Sa Earp.

Since $K(s,t)$ is bounded from above by a negative constant $c$, we can choose $d$
such that $c<d<0$. Then from (\ref{1}) and (\ref{2})
\begin{equation}
\label{4} \left(\frac{G_s}{G}\right)^2+2\left(\frac{G_s}{G}\right)_s>
\left(\frac{\widetilde{G}_s}{\widetilde{G}}\right)^2+
2\left(\frac{\widetilde{G}_s}{\widetilde{G}}\right)_s.
\end{equation}

Now, we observe that given two real functions $f(x),g(x)$ defined on an interval
$I$, with $f(x_0)=g(x_0)$ and satisfying
$$
2f'(x)+f(x)^2>2g'(x)+g(x)^2,
$$
then $f(x)> g(x)$, for all $x> x_0$ on $I$.

Thus, from (\ref{3}), (\ref{7}) and (\ref{4}),
$$
\frac{G_s}{G}>
\frac{\widetilde{G}_s}{\widetilde{G}}=\frac{-2\widetilde{h}_{ss}}{\widetilde{h}_s(1+\widetilde{h}_s^2)},\qquad
s>0,t\in\r
$$
or equivalently, $h=\widetilde{h}$ satisfies the inequality in (\ref{5}). That is,
the graph $\psi(s,t)=(\varphi(s,t),\widetilde{h}(s))$ on $\Omega$ has strictly
positive mean curvature for its unit normal pointing down. This graph has value
$+\infty$ on $\g$ and goes to zero when the distant to $\g$ tends to infinity.

Finally, we obtain the minimal graph with the desired properties as follows. Let
$p\in\g\subseteq \m$, $C(n)$ the geodesic circumference in $\m$ centered at $p$ and
radius $n$, $A(n)=C(n)\cap\Omega$ and $B(n)$ the segment $\g([-n,n])$. Now, consider
the Jordan curve $\Gamma(n)$ in $\m\times\r$ obtained by the arcs $A(n)\times\{0\}$,
$B(n)\times\{n\}$ and the vertical segments joining their end points. Let $\Sigma_n$
be the minimal disk which is a solution of the Plateau problem for $\Gamma(n)$.

$\Sigma_n$ is the  graph of a function $u_n$ on the domain bounded by $A(n)\cup
B(n)$ with $u_{n|A(n)}=0$ and $u_{n|B(n)}=n$. The sequence $\{u_n\}$ is non
decreasing and non negative. In addition, from the comparison  principle, it is
bounded from above by $\widetilde{h}$. Therefore, $\Sigma_n$ converges to a minimal
graph on $\Omega$. This completes the proof of Proposition \ref{megascherk}; hence,
Theorem \ref{JS} as well.
\end{proof}

As in \cite{CR}, we can extend Theorem \ref{JS} to more general domains. We say that
a convex domain $D\subseteq\m$ is {\it admissible} if
\begin{enumerate}
\item the (non empty) finite set $\partial_\infty D$ are the vertices of an ideal
polygon,
\item given two convex arcs $C_1,C_2\subseteq\partial D$ with a common vertex $x\in \partial_\infty D$
there exist two sequences of points $x_n\in C_1$, $y_n\in C_2$ converging to $x$,
such that the distance between $x_n$ and $y_n$ tends to zero.
\end{enumerate}

The second condition in the above definition is used in order to obtain a maximum
principle for minimal graphs over the domain $D$. Moreover, the domain bounded by an
ideal polygon is admissible since the distance between two geodesics with a common
point at infinity goes to zero.

Now, we present a Jenkins-Serrin type theorem when we fix continuous boundary data
on some components of the boundary, whose proof can be shown as in \cite{CR}.

Let $D$ be an admissible domain in $\m$. We seek a solution to the Dirichlet problem
for the minimal surface equation in $D$ which is $+\infty$ on geodesic sides
$A_1,\ldots,A_k$ of $\partial D$, and equals $-\infty$ on other geodesic sides
$B_1,\ldots,B_{k'}$ of $\partial D$,  and equal to continuous functions
$f_i:C_i\fl\r$ on the remaining (nonempty) convex arcs of $\partial D$.

\begin{teo}
There exists a unique solution in $D$ to the above Dirichlet problem if, and only
if,
$$
2a(\P)<|\P|\qquad\mbox{and}\qquad 2b(\P)<|\P|
$$
for all inscribed polygon $\P$ in $D$.
\end{teo}

The uniqueness property in the two previous Theorems is guaranteed by a maximum
principle over admissible domains. The proof when $\m=\h$ was given in \cite{CR} and
it also works in our situation.

\begin{teo}
{\bf (Generalized Maximum Principle)} Let $D\subseteq\m$ be an admissible domain.
Let us consider a domain $\Omega\subseteq D$ and $u,v\in{\cal
C}^0(\overline{\Omega})$ two solutions to the minimal surface equation in $\Omega$
with $u\leq v$ on $\partial \Omega$. Then $u\leq v$ in $\Omega$.
\end{teo}

In addition, as it was explained in \cite{CR}, the previous maximum principle also
applies when the solutions $u,v$ to the minimal surface equation have infinite
boundary values along some geodesic arcs of $\partial D$. When $\partial D$ is made
of complete geodesics and $u$, $v$ take the same infinite value on the whole
boundary of $D$, then $u$ and $v$ agree up to an additive constant.

\section{Extending the domain of a solution.}

The aim of this section is to show that a solution $u$ to the minimal surface
equation on an admissible polygonal domain $D$ with infinity boundary data can be
``extended'' to a larger polygonal domain. By this we mean that a solution $v$ on a
larger domain can be chosen arbitrarily close to $u$ on a fixed compact set $K$ of
$D$.

To establish this result, we first need to show the existence of some special ideal
quadrilaterals. We will call {\it ideal Scherk} surface a graph given by Theorem
\ref{JS}, and {\it Scherk domain} the domain where it is defined.

\begin{pro}\label{is}
Let $x,y,z$ be three points in $\mi$. Let $\g$ be the geodesic joining $x$ and $y$,
and $\Omega$ the connected component of $\m-\g$ such that
$z\not\in\partial_\infty\Omega$. Then there exist a point
$w\in\partial_\infty\Omega$ and an ideal Scherk surface over the domain bounded by
the ideal quadrilateral with vertices $x,y,z,w$. (Cf. Figure 3.)
\end{pro}

\begin{figure}[h]
\mbox{}
\begin{center}
\includegraphics[height=5cm]{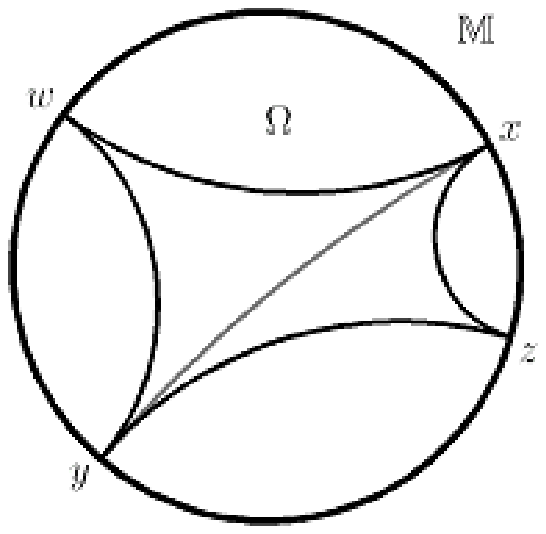}\\
Figure 3.
\end{center}
\end{figure}

We start by establishing some previous Lemmas in order to prove the above
Proposition.

\begin{lem}\label{interseccion}
Let $H_1,H_2$ be two different horocycles in $\m$. Then they intersect at most at
two points.
\end{lem}
\begin{proof}
Let us assume that $H_1, H_2$ are horocycles at $x,y\in\mi$, respectively. If $x=y$
then they do not intersect. So, we can suppose $x\neq y$.

Let $p\in H_1\cap H_2$. Then, the intersection between the two horocycles at $p$ is
transversal, unless $p$ is in the geodesic $\g_{xy}$ joining $x$ and $y$. In the
latter case, if $\gamma_p$ is the geodesic through $p$ tangent to $H_1$ and $H_2$,
then each convex horodisk $B_1$, $B_2$, with respective boundaries $H_1$, $H_2$,
must be on different sides of $\gamma_p$. Therefore,
 each horocycle is in the concave part of the other and the
intersection is only $p$.

Thus, if $H_1\cap H_2$ has more than one point then each intersection is
transversal. Let us consider $p_1,p_2\in H_1\cap H_2$ and ${\cal I}$ the compact arc
in $H_1$ joining $p_1$ and $p_2$. Let $p_0$ be a point in ${\cal I}$ at the largest
distance from $H_2$. Then, the horocycle $\widetilde{H_2}$ at $y$ passing trough
$p_0$ intersects $H_1$ in a tangent way. In particular, $p_0$ is the point
$\g_{xy}\cap H_1$.

So, for any two points $p_1,p_2\in H_1\cap H_2$, we have that $p_0$ is in the
interior of the compact arc of $H_1$ joining $p_1$ and $p_2$. Therefore, there are
at most two points in the intersection between $H_1$ and $H_2$.
\end{proof}
\begin{lem}\label{w}
Let $x,y\in\mi$, and $H_x,H_y$ two disjoint horocycles at $x, y$, respectively.
Consider the open set ${\cal I}$ of points in $\mi$ between $x$ and $y$, where we
assume $\mi$ ordered counter-clockwise. For any point $z$, we define $L:{\cal
I}\fl\r$ as
$$
L(z)=d(H_y,H_z)-d(H_z,H_x),
$$
where $H_z$ is any horocycle at $z$ disjoint from the previous ones and $d(,)$
denotes distance in $\m$. Then, $L$ is a homeomorphism from ${\cal I}$ onto $\r$.
\end{lem}
\begin{proof}
We observe that the definition of the function $L$ does not depend on the chosen
horocycle $H_z$ at $z$, and also makes sense when $H_z$ intersects $H_x$ or $H_y$ in
one point.

Let us consider the homeomorphism $h_1:\mi-\{x\}\fl H_x$ sending each point
$z\in\mi$ different from $x$ to the intersection between $H_x$ and the geodesic
joining $x$ and $z$. Analogously, we consider the homeomorphism $h_2:\mi-\{y\}\fl
H_y$.

Then $L(z)=d(h_2(z),H_z)-d(H_z,h_1(z))={\cal D}(h_1(z),z,h_2(z))$ is a continuous
function, where ${\cal D}$ is given by (\ref{D}).

Now, we see that $L$ is injective. Let $p,q$ be two points in ${\cal I}$ oriented
such that $x<p<q<y$. Let $H_p,H_q$ be the smallest horocycles at $p,q$,
respectively, such that they intersect $H_x\cup H_y$. We distinguish tree cases:
\begin{enumerate}
\item $H_q$ only intersects $H_y$ and $H_p$ only intersects $H_x$,
\item $H_q$ intersects $H_x$,
\item $H_p$ intersects $H_y$.
\end{enumerate}

The case 1 is trivial because $L(q)<0<L(p)$. And, since the cases 2 and 3 are
symmetric, we only need to study case 2.

First, we observe that $H_p$ does not intersects $H_y$. Otherwise, the intersection
holds at $h_2(p)$ and each connected component of $H_p-\{h_2(p)\}$ intersects twice
to $H_q$, which contradicts Lemma \ref{interseccion} (cf. Figure 4).

\begin{figure}[h]
\mbox{}
\begin{center}
\includegraphics[height=5cm]{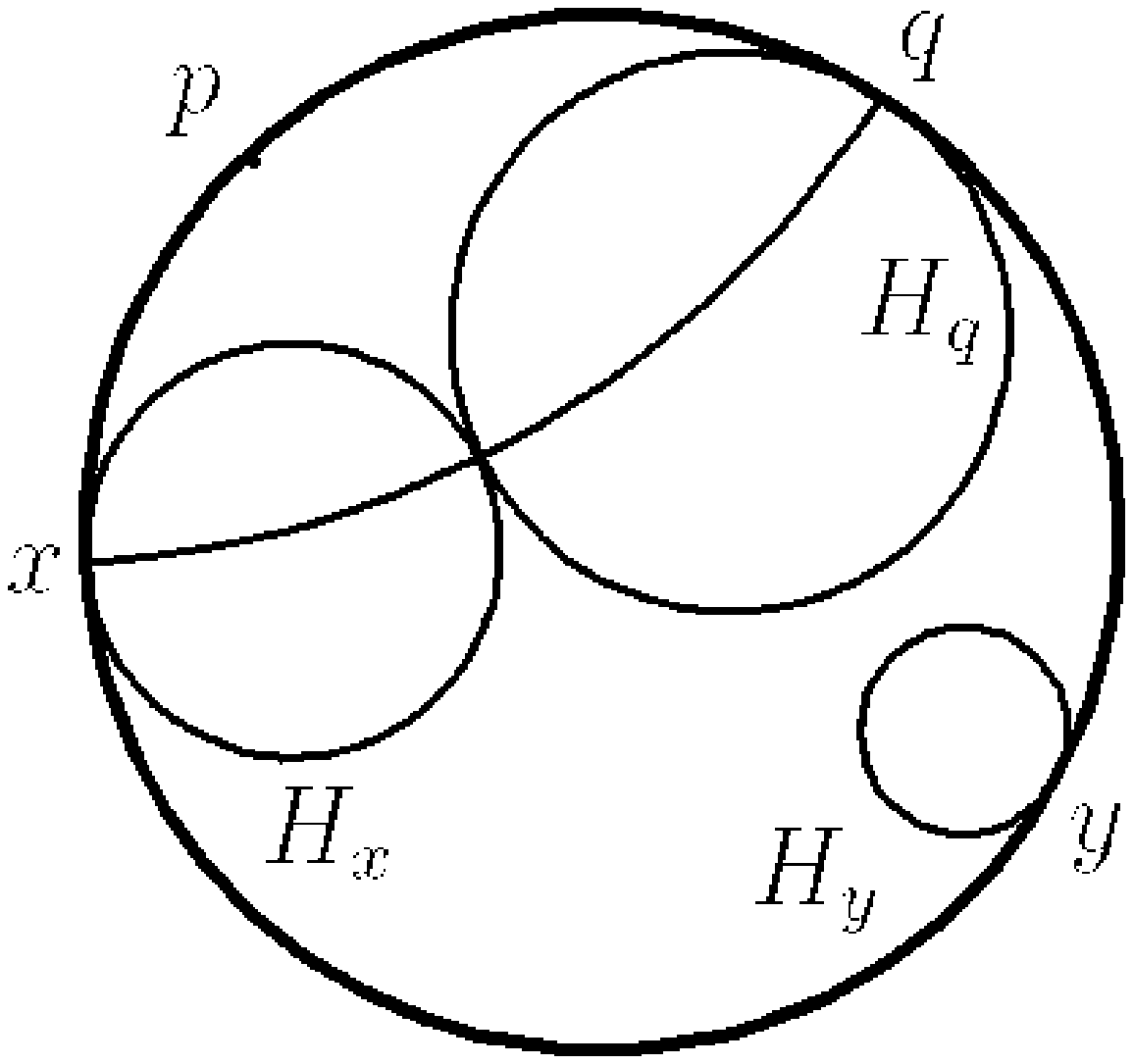}\\
Figure 4.
\end{center}
\end{figure}

Thus, $H_p$ and $H_q$ intersect $H_x$ and we only need to show that
$d(H_q,H_y)<d(H_p,H_y)$ in order to obtain that $L(q)<L(p)$. For that, enlarge $H_y$
to the first horocycle $\widetilde{H_y}$ that intersects $H_p\cup H_q$. From the
previous discussion, replacing $H_y$ and $\widetilde{H_y}$, $H_p$ does not
intersects $\widetilde{H_y}$ (cf. Figure 5). Hence, $d(H_q,H_y)<d(H_p,H_y)$ and the
case 2 is proved.

\begin{figure}[h]
\mbox{}
\begin{center}
\includegraphics[height=5cm]{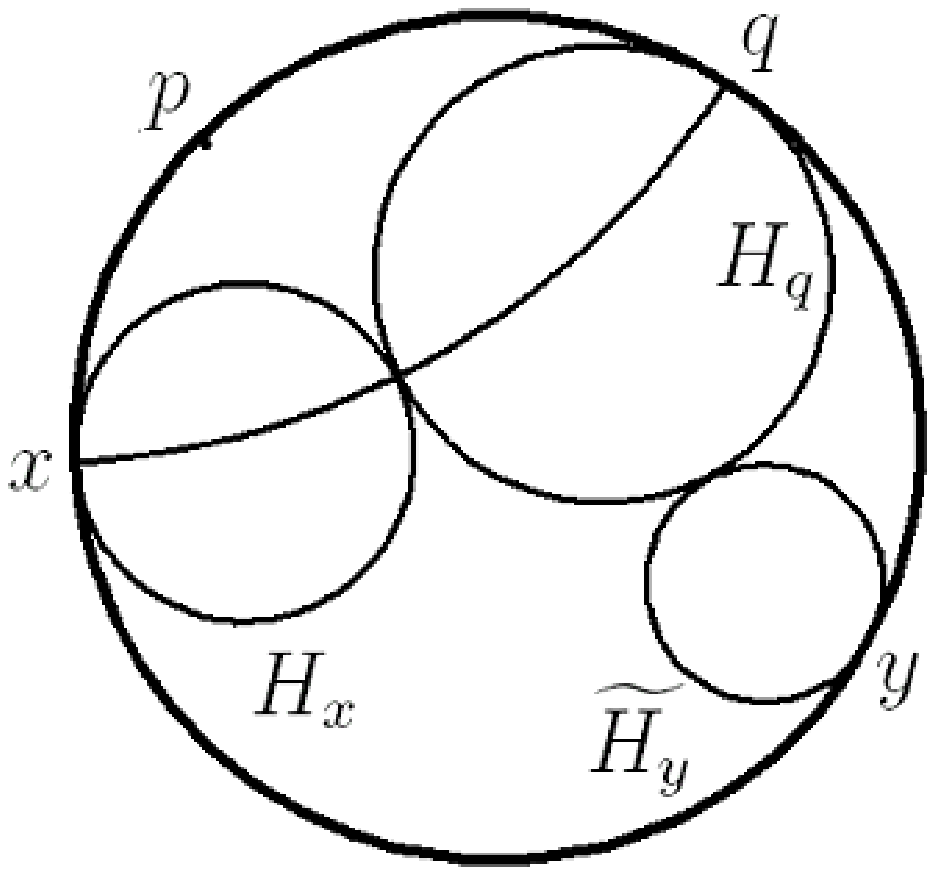}\\
Figure 5.
\end{center}
\end{figure}

Now, let us see that there exists a sequence of points $p_n\in{\cal I}$ such that
$L(p_n)$ goes to $-\infty$.

Let $\g$ be the geodesic joining $x$ and $y$ and $\Omega$ the connected component of
$\m-\g$ which contains ${\cal I}$ in its ideal boundary. Consider $q\in H_y\cap
\Omega$ and $\g_q$ the unique geodesic joining $q$ and $x$. Let us denote by $p$ the
other point of $\g_q$ at the ideal boundary.

The smallest horocycle $H_p$ intersecting $H_x\cup H_y$ does not intersects $H_x$.
Otherwise, $\g_q$ would be contained in the horodisks bounded by $H_x$ and $H_q$.
But, since $q\in\g_q\cap H_y$ then $q$ would be in the horodisk bounded by $H_p$,
which is a contradiction (cf. Figure 6).

\begin{figure}[h]
\mbox{}
\begin{center}
\includegraphics[height=5cm]{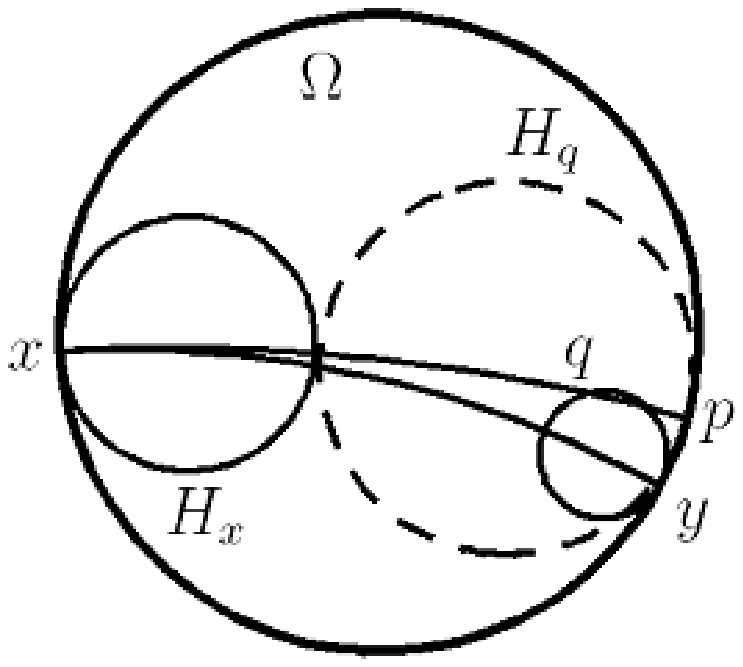}\\
Figure 6.
\end{center}
\end{figure}

Thus $L(p)=-d(H_p,H_x)=-(d(H_p,q)+d(q,H_x))\leq-d(q,H_x)$. Therefore, taking a
sequence $q_n\in H_y\cap \Omega$ converging to $y$, we obtain the corresponding
sequence $p_n$ such that $\lim L(p_n)\leq-\lim d(q_n,H_x)=-\infty.$

Analogously, it can be proved that there exists a sequence of points $p_n\in{\cal
I}$ such that $L(p_n)$ goes to $+\infty$.

Hence, since ${\cal I}$ has the topology of an interval, $L$ is a strictly
monotonous continuous function, and there exist sequences in ${\cal I}$ whose image
tend to $-\infty$ and $+\infty$, then $L$ is a homeomorphism from ${\cal I }$ onto
$\r$.
\end{proof}

Now, let us denote by $|xy|$ the distance between two points in $\m^\ast=\m\cup\mi$,
where we indicate distance between horocycles if $x$ or $y$ are in $\mi$.

\begin{lem}({\bf Generalized Triangle Inequality.})\label{dt}
Consider a triangle with vertices $x_1,x_2$ in $\m^\ast=\m\cup\mi$ and another point
$x_3\in\m$. Then, $$|x_1x_2|\leq|x_1x_3|+|x_3x_2|.$$

Moreover, if $x_1,x_2,x_3\in\mi$ then there exist horocycles at these points such
that the following three inequalities are simultaneously satisfied
$$|x_ix_j|<|x_ix_k|+|x_kx_j|,\qquad \{i,j,k\}=\{1,2,3\}.$$
\end{lem}
\begin{no}
When $x_3\in\m$, the quantity $|x_1x_3|+|x_3x_2|-|x_1x_2|$ does not depend on the
chosen disjoint horocycles, if any. However, it is important to bear in mind that
$|x_ix_k|+|x_kx_j|-|x_ix_j|$ depends on the chosen horocycles if
$x_1,x_2,x_3\in\mi$.
\end{no}
\noindent {\it Proof of Lemma \ref{dt}}. First, we consider the case $x_3\in\m$.

If $x_1,x_2\in\m$ then the inequality is clear. Thus, let us assume $x_1\in\mi$.
Then, enlarge the horocycle $H_{x_1}$  to another horocycle $\widetilde{H_{x_1}}$
which intersects $x_3$ or $x_2$ (or $H_{x_2}$), for the first time.

If $x_2$ (or $H_{x_2}$) intersects $\widetilde{H_{x_1}}$ the inequality is clear.
Otherwise, $x_3\in\widetilde{H_{x_1}}$, and so the distance from ${x_2}$ (or
$H_{x_2}$) to $\widetilde{H_{x_1}}$ is less than or equal to the distance to $x_3$
and the inequality also holds.

Finally, we consider $x_1,x_2,x_3\in\mi$ and three pairwise disjoint horocycles
$H_{x_1},H_{x_2},H_{x_3}$. Now, fix $H_{x_1},H_{x_2}$ and consider a small enough
horocycle $ \widetilde{H_{x_3}}$ such that $|x_1x_2|<|x_1x_3|+|x_3x_2|.$

To obtain the second inequality, we consider a smaller horocycle
$\widetilde{H_{x_2}}$ at $x_2$, if necessary, such that
$|x_1x_3|<|x_1x_2|+|x_2x_3|.$ And now we observe that the first inequality remains
unchanged for the horocycles $H_{x_1},\widetilde{H_{x_2}},\widetilde{H_{x_3}}$.

Following the same process, we take a small horocycle $\widetilde{H_{x_1}}$ at $x_1$
such that $|x_2x_3|<|x_2x_1|+|x_1x_3|.$ Since the previous two inequalities do not
change, the Lemma follows. \hfill$\square$\vspace{5mm}

\noindent{\it Proof of Proposition \ref{is}.} Let $H_z$ be a horocycle at $z$. Take
$H_x, H_y$ disjoint horocycles at $x, y$, respectively, at the same distance from
$H_z$. Then, using Lemma \ref{w}, there exists a point $w\in\partial_\infty\Omega$
such that $d(H_w,H_x)=d(H_w,H_y)$ for any horocycle $H_w$ at $w$ disjoint from $H_x$
and $H_y$. That is, $a(\G)-b(\G)=0$ for the ideal quadrilateral $\G$ with vertices
$x,y,z,w$.

Finally, from the Generalized Triangle Inequality, condition 2 in Theorem \ref{JS}
is satisfied and, so, there exists an ideal Scherk surface over the domain bounded
by $\G$. \hfill$\square$\vspace{5mm}

Now,  we establish some notation.

Given an even set of points $a_0,a_1,\ldots,a_{2n-1}$ in $\m(\infty)$, which we will
assume ordered counter-clockwise, we denote by $\P(a_0,a_1,\ldots,a_{2n-1})$ the
ideal polygon in $\m$ whose vertices are these points.

In order to obtain an ideal Scherk surface on the domain bounded by
$\P(a_0,a_1,\ldots,a_{2n-1})$, we will fix $+\infty$ boundary data on the sides
$[a_{2k},a_{2k+1}]$ and $-\infty$ boundary data on the sides $[a_{2k+1},a_{2k+2}]$.
Here, $[x,y]$ denotes the complete geodesic joining the points $x,y\in\mi$, and we
identify $a_{2n}= a_0$.

\begin{pro}\label{p2}
Let $u$ be an ideal Scherk graph on the domain $D$ bounded by an ideal polygon
$\P(a_0,a_1,a_2,\ldots,a_{2n-1})$. Now, we attach to $D$ two Scherk domains bounded
by $\P(a_0,b_1,b_2,a_1)$ and $\P(a_1,b_3,b_4,a_2)$. Then, given a compact set
$K\subseteq D$ and $\varepsilon>0$, there exists an ideal Scherk graph $v$ on the
domain bounded by the ideal polygon
$\P(a_0,b_1,b'_2,a_1,b'_3,b_4,a_2,\ldots,a_{2n-1})$ such that
$$
\|v-u\|_{{\cal C}^2(K)}\leq\varepsilon
$$
where $b'_2, b'_3$ can be chosen in any punctured neighborhood of $b_2,b_3$ in
$\m(\infty)$, respectively. (Cf. Figure 7.)
\end{pro}

\begin{figure}[h]
\mbox{}
\begin{center}
\includegraphics[height=5cm]{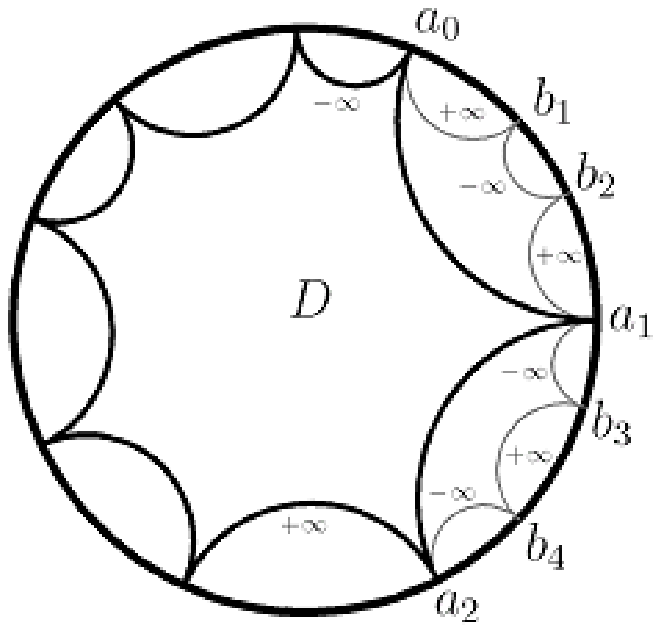}\\
Figure 7.
\end{center}
\end{figure}

\begin{no}
We observe that the existence of the Scherk domains bounded by $\P(a_0,b_1,b_2,a_1)$
and $\P(a_1,b_2,b_3,a_2)$ is guaranteed by Proposition \ref{is}.
\end{no}

The proof of Proposition \ref{p2} proceeds as in \cite{CR}; we will first prove
three lemmas. Following the notation in Proposition \ref{p2}, we denote by $E_1$ the
domain bounded by the polygon $\P(a_0,b_1,b_2,a_1)$, by $E_2$ the domain bounded by
$\P(a_1,b_3,b_4,a_2)$ and by $D_0$ the global domain bounded by
$\G=\P(a_0,b_1,b_2,a_1,b_3,b_4,a_2,\ldots,a_{2n-1})$. Then, it is clear that $\G$
satisfies Condition 1 in Theorem \ref{JS} and, in addition, one obtains
\begin{lem}\label{ofu}
Condition 2 in Theorem \ref{JS} is satisfied by every inscribed polygon in $D_0$,
except the boundaries of $E_1, E_2, D_0-E_1$ and $D_0-E_2$.
\end{lem}
\begin{proof}
It is clear that the boundaries of $E_1, E_2, D_0-E_1$ and $D_0-E_2$ do not satisfy
Condition 2 in Theorem \ref{JS}. Therefore, we start with a inscribed polygon $\P$
in $D_0$ different from them.

From Remark \ref{c2}, we can assume that $\P$ has the adjacent side of $\partial
D_0$ with $+\infty$ data, at any vertex of $\P$. And we only have to prove that
$2a(\P)<|\P|$, for some choise of horocycles at the vertices.

Let $D_¶$ be the domain bounded by $\P$ and $\P'$ the boundary of $D_¶-E_2$. $\P'$
is a polygon with some possible vertices in $\m$ (cf. Figure 8). We are going to
show that if $2a(\P')<|\P'|$ then $2a(\P)<|\P|$. And, so, we will only need to prove
that the inequality is true for $\P'$.

\begin{figure}[h]
\mbox{}
\begin{center}
\includegraphics[height=5cm]{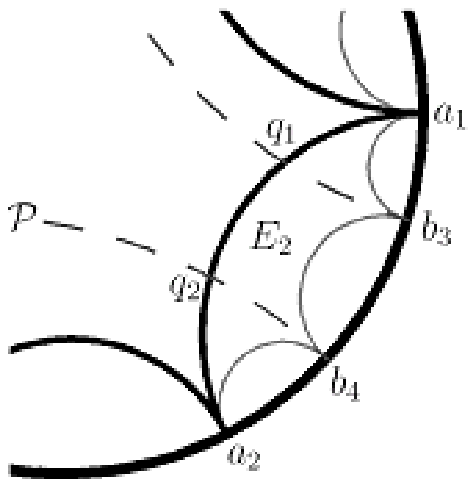}\\
Figure 8.
\end{center}
\end{figure}

If the geodesic $[b_3,b_4]$ does not belong to $\P$ then $\P=\P'$ and the result is
obvious. Let $d_1$ be the vertex of $\P$  previous to $b_3$, and $d_2$ the vertex
following $b_4$. Consider $q_1=[d_1,b_3]\cap[a_1,a_2]$ and
$q_2=[b_4,d_2]\cap[a_1,a_2]$ (cf. Figure 8). Observe that $q_i$ could be $a_i$; in
that case we have $|a_iq_i|=0$. We have,
$$
\begin{array}
{l} a(\P)=a(\P')+|b_3b_4|,\\
|\P|=|\P'|-|q_1q_2|+|q_1b_3|+|b_3b_4|+|b_4q_2|.
\end{array}
$$
On the other hand, since $E_2$ is a Scherk domain we can assume
$d_0=|a_1b_3|=|b_3b_4|=|b_4a_2|=|a_1a_2|$. Thus, from the Generalized Triangle
Inequality, $|a_1b_3|\leq |a_1q_1|+|q_1b_3|$, $|a_2b_4|\leq |a_2q_2|+|q_2b_4|$, and
using that $2a(\P')<|\P'|$, we have
\begin{eqnarray*}
|\P|-2a(\P)&>&|q_1b_3|+|b_4q_2|-|q_1q_2|-|b_3b_4|\\
&=&|q_1b_3|+|b_4q_2|-2|b_3b_4|+|a_1q_1|+|a_2q_2|\\
&=&(|q_1b_3|+|a_1q_1|-|a_1b_3|)+(|b_4q_2|+|a_2q_2|-|a_2b_4|)\geq 0.
\end{eqnarray*}

Therefore, in order to finish the proof of Lemma \ref{ofu} we only need to see that
$2a(\P')<|\P'|$.

Let  $D_{\P'}$ be the domain bounded by $\P'$, and $\P''$ the boundary of
$D_{\P'}-E_1$. We use a flux inequality, for the initial minimal graph $u$, over the
domain bounded by $\P''$ to obtain the desired inequality.

From the minimal graph equation (\ref{minimal}), the field $X=(\nabla u)/W$, with
$W=\sqrt{1+|\nabla u|^2}$ is divergence free. We write $\P''$ as the union of three
sets: the union of all geodesic arcs of $\P''$ with boundary data $+\infty$ and
disjoint from $[a_0,a_1]$, $I_1=\P''\cap[a_0,a_1]$ and $J$ the union of the
remaining arcs.

If $\nu$ is the unit outward normal to $\P''$ then $X=\nu$ on the sides with
$+\infty$ boundary data (see \cite{CR,P}), and so
$$
0=F_u(\partial \P'')=a(\P'')+|I_1|+F_u(J)+\rho,
$$
where, for instance, $F_u(\partial \P'')=\int_{\partial \P''}\me{X}{\nu}$ denotes
the flux of $u$ along $\partial \P''$.

Here, the flux $F_u(J)$ is taken on the compact arcs of $J$ outside the horocycles,
and the number $\rho$ corresponds to the remaining flux of $X$ along some parts of
horocycles.

Since $|\P''|=a(\P'')+|I_1|+|J|$, we have
\begin{equation}
\label{titotitati}
|\P''|-2\ (a(\P'')+|I_1|)=|J|+F_u(J)+\rho.
\end{equation}

We are assuming that $\P$ has the adjacent side of $\partial D_0$ with $+\infty$
data, for any vertex of $\P$. Therefore, the estimation of $|\P'|-2a(\P')$ does not
depend on the chosen horocycles. Moreover, $\P''$ is not empty and it is different
from the boundary of $D$. To see that, observe that if $\P''$ is empty then the
previous condition on $\P$ implies that $\P$ is the boundary of $E_1$, and if $\P''$
is the boundary of $D$ then $\P$ is the boundary of $D\cup E$ or $D_0$.

Hence, $\P''$ has interior arcs in $D$, and $F_u(\alpha)+|\alpha|$ is positive on
each interior arc $\alpha$. Thus, $|J|+F_u(J)$ is positive and non decreasing when
we choose smaller horocycles. In addition, we can select these horocycles so that
$|\rho|$ is as small as desired. Therefore, from (\ref{titotitati}), we can assume
\begin{equation}
\label{fl}
|\P''|-2\ (a(\P'')+|I_1|)>0
\end{equation}
for suitable horocycles at the vertices.

Now, we show that $2a(\P')<|\P'|$. For that, we distinguish three cases
\begin{enumerate}
\item $a_0$ and $a_1$ are vertices of $\P$,
\item only one of the points $a_0,a_1$ is a vertex of $\P$,
\item neither $a_0$ nor $a_1$ are vertices of $\P$.
\end{enumerate}

In the case 1, $[a_0,b_1]$ and $[b_2,a_1]$ must be sides of $\P$. So, $[b_1,b_2]$ is
also a side of $\P$ and $E_1$ is contained in the domain bounded by $\P$. Thus,
$$
a(\P')=a(\P'')+2d_0,\qquad |\P'|=|\P''|+2d_0\quad\mbox{and}\quad |I_1|=d_0.
$$
And (\ref{fl}) gives us $2a(\P')<|\P'|$.

For the case 2, we assume for instance that $a_0$ is a vertex of $\P$, but not
$a_1$. Then $b_2$ is not a vertex of $\P$ and we consider the point $q$ which is the
intersection between the geodesic $[a_0,a_1]$ and the geodesic joining $b_1$ and the
following vertex of $\P$ (cf. Figure 9). Then $I_1=[a_0,q]$ and
$$
a(\P')=a(\P'')+d_0,\qquad |\P'|=|\P''|-|I_1|+d_0+|b_1q|.
$$

\begin{figure}[h]
\mbox{}
\begin{center}
\includegraphics[height=5cm]{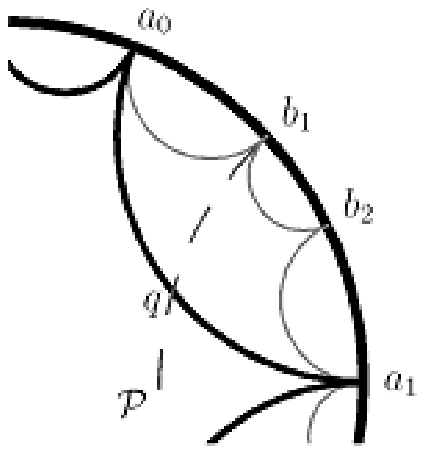}\\
Figure 9.
\end{center}
\end{figure}

And using (\ref{fl})
\begin{eqnarray*}
0&<&(|\P'|+|I_1|-d_0-|b_1q|)-2(a(\P')-d_0+|I_1|)\\
&=&|\P'|-2a(\P')+d_0-|I_1|-|b_1q|.
\end{eqnarray*}
Therefore, using the Generalized Triangle Inequality for the triangle with vertices
$a_0,b_1,q$ we have $d_0-|I_1|-|b_1q|\leq 0$ and the inequality $2a(\P')<|\P'|$ is
proved.

Case 3 is clear because the domain bounded by $\P'$ lies on $D$ and the flux formula
(\ref{fl}) gives the desired inequality.
\end{proof}

Now we will perturb $E_1$ and $E_2$ to obtain a Scherk domain. We will do this so
that $E_1, E_2$ and their complements are inscribed polygons satisfying Condition 2
of Theorem \ref{JS}. By the previous Lemma \ref{ofu}, the other inscribed polygons
in $D_0$ also satisfy Condition 2. Hence, if we make the perturbation of $E_1,E_2$
small enough, the strict inequalities (Condition 2) satisfied by these inscribed
polygons will remain strict inequalities. We now make this precise.

\begin{lem}
For any punctured neighborhoods of $b_2,b_3$ in $\mi$, we can choose, respectively,
two points $b'_2,b'_3$ such that the domain bounded by
$\P(a_0,b_1,b'_2,a_1,b'_3,b_4,a_2,\ldots,a_{2n-1})$ is, in fact, a Scherk domain.
\end{lem}
\begin{proof}
We first observe that
\begin{eqnarray*}
|a_0a_1|-|a_1b_2|+|b_2b_1|-|b_1a_0|&=&0\\
|a_2a_1|-|a_1b_3|+|b_3b_4|-|b_4a_2|&=&0
\end{eqnarray*}
since $E_1$ and $E_2$ are Scherk domains.

Now, using Lemma \ref{w}, there exist unique $b_2(t)$, $b_3(t)$ such that
\begin{eqnarray}
t&=&|a_0a_1|-|a_1b_2(t)|+|b_2(t)b_1|-|b_1a_0|\nonumber\\[-3mm]
&&\label{t}\\[-3mm]
&=&|a_2a_1|-|a_1b_3(t)|+|b_3(t)b_4|-|b_4a_2|,\nonumber
\end{eqnarray}
where $b_2(t)$ varies in the open interval between $b_1$ and $a_1$ at infinity,
$b_3(t)$ between $a_1$ and $b_4$, and $t\in\r$. In addition, the functions $b_i(t)$
are homeomorphisms.

Let $\G(t)=\P(a_0,b_1,b_2(t),a_1,b_3(t),b_4,a_2,\ldots,a_{2n-1})$. From (\ref{t}),
Condition 1 in Theorem \ref{JS} is satisfied for any $t\in\r$. In addition, if $t>0$
then the domains $E_1(t)$ and $E_2(t)$ bounded by $\P(a_0,b_1,b_2(t),a_1)$,
$\P(a_1,b_3(t),b_4,a_2)$ and their complements also satisfy Condition 2. In order to
obtain Condition 2 for the other inscribed polygons we argue as follows.

From Lemma \ref{ofu}, every polygon inscribed in the domain bounded by $\G(0)$
satisfies Condition 2, except $E_1(0)$, $E_2(0)$ and their complements.  Observe
that the inequalities in Condition 2 are strict, and the number of inscribed
polygons is finite. From Lemma \ref{w}, these inequalities depend continuously on
$b_2(t)$ and $b_3(t)$, so one has that there exists $t_0>0$ such that Condition 2 is
also satisfied for any domain bounded by $\G(t)$, with $0<t<t_0$.
\end{proof}
\noindent{\it Proof of Proposition \ref{p2}.} The proof is a verification that the
arguments in \cite{CR} work in our context. Let us denote by $D_t$ the Scherk domain
bounded by $\G(t)$. Consider the graph of a Scherk surface $u_t$ defined on $D_t$
with the corresponding infinite boundary data. First, we show that $\nabla u$ is the
limit of $\nabla {u_t}_{|D}$ when $t$ goes to zero.

Let us consider the divergence free fields $X_t=(\nabla {u_t})/W_t$ and $X=(\nabla
u)/W$, with $W_t=\sqrt{1+|\nabla u_t|^2}$, $W=\sqrt{1+|\nabla u|^2}$, associated to
$u_t$ and $u$, respectively. We now see that $X_t$ converges to $X$ on $D$ when $t$
tends to zero.

Consider the outer pointing normal $\nu$ along the boundary of $D$. We have fixed
the same infinite boundary data on $\partial D-([a_0,a_1]\cup[a_1,a_2])$, so
$X_t=X=\pm\nu$ on this set.

On the boundary of $E_1(t)$ truncated by the horocycles, the flux of $X_t$ is zero.
Hence,
$$
0=|a_0b_1|-|b_1b_2(t)|+|b_2(t)a_1|+\int_{[a'_0,a'_1]}\me{X_t}{-\nu}+F_{u_t}(I_t),
$$
where $[a'_0,a'_1]$ is the compact geodesic arc between the horocycles at $a_0$ and
$a_1$, and $I_t$ is the set of arcs included in the four horodisks. Then, from
(\ref{t}),
$$
t=\int_{[a'_0,a'_1]}(1-\me{X_t}{\nu})+F_{u_t}(I_t),
$$
and taking limits for smaller horocycles at the vertices one has the convergence of
the integral on the whole geodesic and
$$
t=\int_{[a_0,a_1]}(1-\me{X_t}{\nu})=\int_{[a_0,a_1]}\me{X-X_t}{\nu},
$$
since $X=\nu$ on $[a_0,a_1]$.

Analogously, one has
$$
t=-\int_{[a_1,a_2]}\me{X-X_t}{\nu}.
$$

Thus, for any family $\alpha$ of disjoint arcs of $\partial D$
\begin{equation}\label{cuatro}
\left|\int_\alpha\me{X-X_t}{\nu}\right|\leq\int_{[a_0,a_1]\cup[a_1,a_2]}\left|\me{X-X_t}{\nu}\right|=2t.
\end{equation}


Now, we study the behavior of the field $X-X_t$ on the interior of $D$. Let $\Si$ be
the graph of $u$ and $\Si_t$ the graph of $u_t$. These graphs are stable, complete
and satisfy uniform curvature estimates by Schoens' curvature estimates. Thus,
$$
\forall \varepsilon>0\ \ \exists\rho>0\mbox{ such that }\forall p\in D\ \ \forall
q\in\Si_t\cap B((p,u_t(p)),\rho)\mbox{ one has }\|N_t(p)-N_t(q)\|\leq\varepsilon.
$$
Here, $\rho$ does not depend on $t$, and $N_t$ denotes the normal to $\Si_t$
pointing down and $B((p, v_t(p)), \rho)$ the ball of radius $\rho$, centered at $(p,
v_t (p))\in\m\times\r$. These estimates remain true for $\Si$.

Therefore, one obtains that fixed $\varepsilon>0$ and $p\in D$ there exists
$\rho_1\leq\rho/2$, which depends continuosly on $p$ but does not depend on $t$,
such that for every $q$ in the disk $B(p,\rho_1)$ in $\m$ with center $p$ and radius
$\rho_1$, we have $|u(q)-u(p)|\leq \rho/2$.

Let us assume now that $\|N_t(p)-N(p)\|\geq 3\,\varepsilon$. Consider the connected
component $\Omega_t(p)$ of $\{q\in D:u(q)-u_t(q)>u(p)-u_t(p)\}$ with $p$ in its
boundary, and $\Lambda_t$ the component of $\partial \Omega_t(p)$ containing $p$.
Since $\Lambda_t$ is a level curve of $u-u_t$ then it is piecewise smooth. Let
$\sigma\subseteq\Si$, $\sigma_t\subseteq\Si_t$ be the two parallel curves which
project on $\Lambda_t\cap B(p,\rho_1)$.

For the points of $\sigma$, we have that if $q\in\Lambda_t\cap B(p,\rho_1)$ then
$|(q,u(q))-(p,u(p))|\leq\rho_1+\rho/2\leq\rho$ and so
$\|N(q)-N(p)\|\leq\varepsilon.$ The same is also true on the parallel curve
$\sigma_t$, that is, $\|N_t(q)-N_t(p)\|\leq\varepsilon$ for all $q\in\Lambda_t\cap
B(p,\rho_1)$.

Thus, using these inequalities and the assumption on the normals at $p$, we obtain
for all $q\in\Lambda_t\cap B(p,\rho_1)$ that
$\|N(q)-N_t(q)\|\geq\|N_t(p)-N_t(q)\|-2\,\varepsilon\geq\varepsilon$.

From Assertion 2.2 in \cite{P},
\begin{equation}\label{Pi}
\me{X-X_t}{\eta}_\m\geq\frac{\|N-N_t\|^2}{4}
\end{equation}
 with
$\eta=\nabla(u-u_t)/|\nabla(u-u_t)|$ orienting the level curve $\Lambda_t$ at its
regular points (see also \cite{CR}). Thus, one has
$$
\int_{\Lambda_t\cap B(p,\rho_1)}\me{X-X_t}{\eta}\geq\frac{\rho_1\varepsilon^2}{2}.
$$

In addition, from (\ref{Pi}), $\me{X-X_t}{\eta}$ is non negative outside the
isolated points where $\nabla (u-u_t)=0$, and so, for every compact arc
$\beta\subseteq\Lambda_t$ containing $\Lambda_t\cap D(p,\rho_1)$ we have
\begin{equation}\label{cinco}
\int_{\beta}\me{X-X_t}{\eta}\geq\frac{\rho_1\varepsilon^2}{2}.
\end{equation}

By the maximum principle, $\Lambda_t$ is not compact in $D$. And, since $\Lambda_t$
is proper on $D$, its two infinite branches go close to $\partial D$. Then there
exists a connected compact part $\beta$ of $\Lambda_t$ , containing $\Lambda_t\cap
B(p,\rho_1)$, and two arcs $\delta$ in $D$ small enough and joining the extremities
of $\beta$ to $\partial D$. Eventually truncating by a family of horocycles, the
flux formula for $X-X_t$ yields
$$
0=\int_\beta
\me{X-X_t}{-\eta}+\int_\alpha\me{X-X_t}{\nu}+F_{u-u_t}(\delta\cup\delta'),
$$
where $\alpha$ is contained in $\partial D$ and $\delta'$ is contained in the
horocycles and correctly oriented. Using (\ref{cuatro}) and (\ref{cinco}) we obtain
$$
\frac{\rho_1\varepsilon^2}{2}\leq 2t+F_{u-u_t}(\delta\cup\delta').
$$

When the length of $\delta\cup\delta'$ goes to zero, one has
$$
\frac{\rho_1\varepsilon^2}{4}\leq t.
$$

Hence, if $t\leq(\rho_1\varepsilon^2)/4$ then
$\|X(p)-X_t(p)\|\leq\|N(p)-N_t(p)\|\leq 3\,\varepsilon$. Since $\rho_1$ only depends
continuously on $p$, this gives us the desired convergence of $X_t$ to $X$ when $t$
goes to zero.

After the normalization $u_t(p_0) = u(p_0)$ for a fixed $p_0\in D$, we have that
$\lim_{t\rightarrow0} {u_t}_{|D} = u$. The convergence is uniform in relatively
compact domains $\widetilde{D}$ of $D$ and ${\cal C}^\infty$ on compact sets of
$\widetilde{D}$. Hence, given a compact set $K\subseteq D$ and $\varepsilon>0$,
there exists a $t$ small enough such that $\|u_t-u\|_{{\cal
C}^2(K)}\leq\varepsilon$. \hfill$\square$

\section{Entire minimal graphs.}

We now establish our main result.

\begin{teo}
Let $\m$ be a Hadamard surface with Gauss curvature bounded from above by a negative
constant. Then, there exist harmonic diffeomorphisms from the complex plane onto
$\m$.
\end{teo}
\begin{proof} The vertical projection from a minimal surface $\Si\subseteq\mm$
into $\m$ is a harmonic map. Therefore, in order to prove the Theorem, we only need
to show that there exist entire minimal graphs in $\mm$ with the conformal structure
of the complex plane.

Let us fix a point $p_0$ in a Scherk domain $D_1\subseteq \m$ and also a compact
disk $K_1\subseteq D_1$. Observe that the existence of $D_1$ is guaranteed by
Proposition \ref{is}.

Consider the homeomorphism $h$ from the set $\s^1_{p_0}$ of unit tangent vectors at
$p_0$ onto $\mi$, which maps a vector $v\in \s^1_{p_0}$ to the point in $\mi$ given
by $\gamma_v(+\infty)$. Here, $\gamma_v(t)$ is the unique geodesic in $\m$ with
initial conditions $\gamma_v(0)=p_0$ and $\gamma'_v(0)=v$. We will measure the angle
between two points $x,y\in\mi$ as the angle between the vectors $h^{-1}(x),
h^{-1}(y)\in\s^1_{p_0}$.

Now, fix a sequence  of positive numbers $\varepsilon_n$ such that $\sum_{n\geq
1}\varepsilon_n<\infty$. We show the existence of an exhaustion of $\m$ by Scherk
domains $D_n$ and by compact disks $K_n\subseteq D_n$ such that each $K_n$ is
contained in the interior of $K_{n+1}$ and a sequence of minimal graphs $u_n$ on
$D_n$ satisfying
\begin{enumerate}
\item $\|u_{n+1}-u_n\|_{{\cal C}^2(K_n)}<\varepsilon_n$,
\item the conformal modulus of the minimal annulus on $K_{i+1}-\mbox{int}(K_i)$ for the graph $u_{n}$ is greater than
one for each $1\leq i\leq n-1$, where $\mbox{int}(K_i)$ denotes the interior of
$K_i$,
\item the angle between two consecutive vertices of the ideal polygon $\partial D_n$
is less than $\pi/2^{n-1}$.
\end{enumerate}

The third condition is clear for $n=1$ since $p_0\in D_1$. Thus, we assume that
there exists the sequence $(D_i,u_i,K_i)$ satisfying the three previous conditions
for $1\leq i\leq n$ and we obtain $(D_{n+1},u_{n+1},K_{n+1})$.

Let $x,y$ be the vertices of a side of $\partial D_n$, and ${\cal I}$ the arc
between $x$ and $y$ that contains no other vertex of $\partial D_n$. We choose the
unique point  $z\in{\cal I}$ such that the angle between $x$ and $z$ agrees with the
angle between $y$ and $z$. Thus, the angle between $x$ and $z$ is less than
$\pi/2^{n}$. Now, from Proposition \ref{is}, there exists $w\in{\cal I}$ such that
the domain bounded by the quadrilateral with vertices $x,y,z,w$ is a Scherk domain.
Moreover, the angle between two consecutive vertices is less than $\pi/2^{n}$.

We attach to each side of $\partial D_n$ an ideal quadrilateral constructed as
above. Then we use Proposition \ref{p2} and perturb all the pairs of sides of
$\partial D_n$ to obtain an ideal Scherk graph $u_{n+1}$ on a larger domain
$D_{n+1}$. This perturbation of the vertices can be done as small as necessary so
that Conditions 1, 3 are satisfied, and also Condition 2 for $1\leq i<n$.

Now, we use the following Lemma. We refer the reader to \cite{CR} for its proof.
\begin{lem}
Every ideal Scherk surface is conformally equivalent to the complex plane.
\end{lem}

Hence, the minimal graph $\Si$ of $u_{n+1}$ is conformally the complex plane. Let
$\Si_0\subseteq\Si$ be the graph of $u_{n+1}$ on the interior of $K_n$. Thus, we can
choose a closed disk $\Si_1\subseteq \Si$ containing $\Si_0$ in its interior such
that the conformal modulus of $\Si_1-\Si_0$ is greater than one. Then, we take
$K_{n+1}$ as the projection of $\Si_1$. In addition, we can enlarge $K_{n+1}$, if
necessary, in such a way that $ K_{n+1}$ contains $\widehat{D}_{n+1}\cap B(p_0,n)$,
where $\widehat{D}_{n+1}$ is the set of points in $D_{n+1}$ a distance greater than
1 to its boundary and $B(p_0,n)$ the geodesic disk centered at $p_0$ of radius $n$.
Thus, Condition 2 is also satisfied.

Observe now that $\m=\cup_{n\geq1}D_n$. This is a straightforward consequence of
Condition 3, since the set of vertices of the domains $D_n$ is dense in $\mi$. In
addition, from the condition between the distance of $\partial K_n$ and $\partial
D_n$ one has that $\m=\cup_{n\geq1}K_n$.

Once we have obtained the previous sequence, we can get the desired entire minimal
graph. Since $u_n(p)$ is a Cauchy sequence for any $p\in\m$, we obtain an entire
minimal graph $u$. On the other hand, on each compact set $K_{i+1}-\mbox{int}(K_i)$
the sequence $u_n$ converges uniformly to $u$ in the ${\cal C}^2$--topology. Hence,
the conformal modulus of the minimal graph of $u$ on $K_{i+1}-\mbox{int}(K_i)$ is at
least one. So, using the Gr\"otzsch Lemma \cite{V}, the conformal type of the
minimal graph of $u$ is the complex plane.
\end{proof}

We also construct harmonic diffeomorphisms from the unit disk onto $\m$ by solving a
Dirichlet problem at infinity.
\begin{teo}\label{chuliperi}
Let $\Upsilon$ be a continuous Jordan curve in the cylinder $\mi\times\r$, which is
a vertical graph. Then, there exists a unique entire minimal graph on $\m$ having
$\Upsilon$ as its asymptotic boundary. Moreover, the conformal structure of this
graph is that of the unit disk.
\end{teo}
\begin{proof}
Let $\varphi:\mi\fl\r$ be the continuous function whose graph is $\Upsilon$. Let us
fix a point $p_0\in\m$. Consider for any unit tangent vector $v$ at $p_0$, the
unique geodesic $\gamma_v(t)$ satisfying $\gamma_v(0)=p_0$ and $\gamma'_v(0)=v$, and
$h:\s^1_{p_0}\fl\mi$ the homeomorphism given by $h(v)=\gamma_v(+\infty)$.

For the continuous function $\varphi\circ h:\s^1_{p_0}\fl\r$, we consider a sequence
of ${\cal C}^2$--functions $\varphi_n:\s^1_{p_0}\fl\r$ converging uniformly to
$\varphi\circ h$. Then, for any positive integer $n$ we consider the graph on the
geodesic circle centered at $p_0$ of radius $n$ given by the curve
$\Upsilon_n(v)=(\gamma_v(n), \varphi_n(v))$, $v\in\s^1_{p_0}$.

Let $\Si_n$ be the minimal surface in $\mm$ obtained as the Plateau solution with
boundary $\Upsilon_n$. The surface $\Si_n$ can be seen as a graph $u_n$ on the
geodesic disk centered at $p_0$ and of radius $n$, by Rado's theorem. Since the
horizontal slices are minimal surfaces, from the maximum principle, the sequence
$\{u_n\}$ is uniformly bounded on compact subsets of $\m$. Thus there is a
subsequence converging to a entire minimal solution $u:\m\fl\r$, uniformly on
compact subsets of $\m$. Let $\Sigma$ be the entire minimal graph given by $u$.

We now prove that the asymptotic boundary of $\Sigma$ is $\Upsilon$. For that,
observe that we only need to show that if $q$ is a point in $\mi\times\r$ such that
$q\notin \Upsilon$ then $q$ does not belong to the asymptotic boundary of $\Sigma$.

Consider $q=(x_0,r)\in\mi\times\r$. We assume, for instance, $r>\varphi(x_0)$. Take
$\varepsilon=(r-\varphi(x_0))/2>0$ and $v_0=h^{-1}(x_0)$. Then, from the uniform
convergence of $\varphi_n$ to $\varphi\circ h$ and the continuity of $\varphi$, we
can assure the existence of $\delta >0$ and $n_0$ such that for all $w\in\s^1_{p_0}$
with $\|w-v_0\|\leq\delta$ and $n\geq n_0$
$$
|\varphi_n(w)-\varphi( h(v_0))|\leq\varepsilon.
$$

Let $w_1,w_2\in\s^1_{p_0}$ be the unit vectors at a distance $\delta$ from $v_0$.
Let $\Omega\subseteq\m$ be the halfspace determined by the geodesic $\alpha$ joining
the points at infinity $h(w_1)$ and $h(w_2)$ and having $x_0$ in
$\partial_{\infty}\Omega$ (cf. Figure 10). From Proposition \ref{megascherk}, there
exits a Scherk type graph $v$ on the halfspace $\Omega$ with boundary data $+\infty$
on $\alpha$ and $(r+\varphi(x_0))/2$ on $\partial_{\infty}\Omega$.

\begin{figure}[h]
\mbox{}
\begin{center}
\includegraphics[height=5cm]{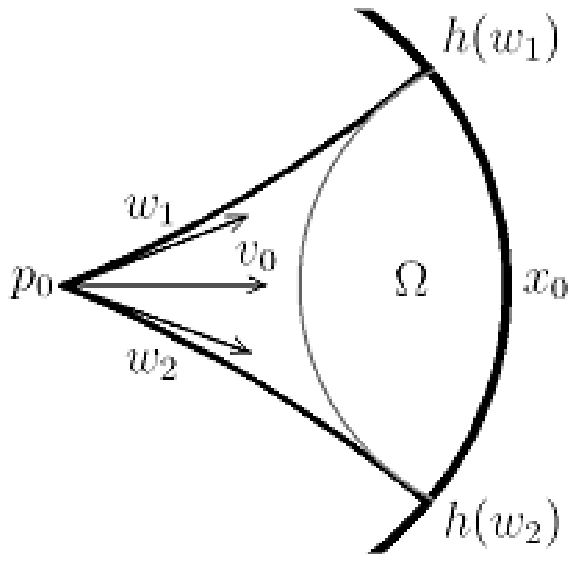}\\
Figure 10.
\end{center}
\end{figure}

From the maximum principle, $u_n\leq v$ on $\Omega$ for all $n\geq n_0$. In
particular, $q=(x_0,r)$ does not belong to the asymptotic boundary of the entire
graph $\Sigma$. Thus, the asymptotic boundary of $\Sigma$ is $\Upsilon$.

The uniqueness part of the Theorem is a straightforward consequence of the maximum
principle. In addition, since the height function is harmonic  and bounded for the
entire minimal graph $\Si$, then its conformal structure must be that of the unit
disk.
\end{proof}

\end{document}